\documentclass[11pt]{amsart}

\usepackage{amsmath,amssymb,amsfonts}
\usepackage{pdfpages}
\usepackage{listings}
\usepackage{graphicx}

\newtheorem{theorem}{Theorem}

\newtheorem*{problem*}{Problem}

\newtheorem*{conj}{Conjecture}

\newcommand{\IR}{\mathbb{R}}

\usepackage[top=2.5cm, bottom=2.5cm, left=2.5cm, right=2.5cm]{geometry}

\usepackage{hyperref}

\hypersetup{
  pdftitle={\title},
  pdfauthor={\author},
  colorlinks=true,
  breaklinks=true,
  bookmarksopen=true,
  bookmarksnumbered=true,
  pdfpagemode=UseOutlines,
  plainpages=false,
  linkcolor=blue,
  citecolor=orange,
  anchorcolor=green}

\begin{document}

\title{On the Koebe Quarter Theorem for Polynomials}

\author{Jimmy Dillies}
\address{Department of Mathematics,  Georgia Southern University, Statesboro GA, 30460, USA}
\email{jdillies@georgiasouthern.edu}
\urladdr{http://jimmy.klacto.net/}

\author{Dmitriy Dmitrishin} 
\address{Department of Applied Mathematics, Odessa National Polytechnic University,  Odessa, 65044, Ukraine}
\email{dmitrishin@opu.ua}

\author{Andrey Smorodin}
\address{Department of Applied Mathematics, Odessa National Polytechnic University,  Odessa, 65044, Ukraine}
\email{andrey.v.smorodin@gmail.com}

\author{Alex Stokolos}
\address{Department of Mathematics,  Georgia Southern University, Statesboro GA, 30460, USA}
\email{astokolos@georgiasouthern.edu}


\date{}

\begin{abstract}
D.~Dimitrov \cite{D} has posed the problem of finding polynomials that set the sharpness of the Koebe Quarter Theorem for polynomials and asked whether Suffridge polynomials \cite{S} are optimal. 
We disprove Dimitrov's conjecture for polynomials of degree 3, 4, 5 and~6. For polynomials of degree 1 and 2 the conjecture is obviously true. 
On the way we introduce a new family of polynomials that allows us to state a conjecture about the value of the Koebe radius for polynomials of a specific degree. 
This article is a continuation of the research started in \cite{DDS}.
\end{abstract}

\maketitle


\section{Introduction}  

Geometric complex analysis has arisen from two fundamental statements: the Koebe Quarter Theorem and the Bieberbach Conjecture. 
{ Koebe's theorem} states that  for any function $f\in \mathcal S$ the image $f(\mathbb D)$ contains a disc of radius 1/4, where $\mathbb D=\{|z|<1\}$ and ${\mathcal S}=\{f(z): f(0)=0,\, f^\prime(0)=1,\, f(z) \ \text{is univalent in}\ \mathbb D\}$. 
The Bieberbach Conjecture says that $|a_k|\le k$ for all $k=1,2,\ldots$ where $f\in \mathcal S$ and the $a_k$ are the Taylor coefficients of $f$. 
For many decades this conjecture was a driving force of the development of geometric complex analysis. Many outstanding mathematicians contributed through partial solutions until it was resolved in full generality by Louis de Brange in 1984.

Both theorems offer sharp bounds as the so-called {\it Koebe function} 
\begin{equation}\notag
\mathcal K(z):=\frac z{(1-z)^2}=z+2z^2+3z^3+\cdots,\qquad z\in\mathbb D,
\end{equation}
is an extremizer for both statements. 
One can see that the radius $\frac14$ in Koebe's theorem is optimal, as also is the estimate $|a_k|\le k$ in de Brange's result. 

A natural question is whether the constant 1/4 as well as the estimate $|a_k|\le k$ can be improved for  polynomials of a specific degree and what would be a polynomial analogue of the Koebe function. 
Say, for polynomials of the first degree the constant is trivially 1; a simple computation demonstrates that for polynomials of degree 2 it is 1/2. 
The task was formalized by Dimitrov who asked:

\begin{problem*}{\cite[Problem 5]{D}}
For any $N\in\mathbb Z_+$, find a polynomial $p_N(z)=z+a_2z^2+...+a_Nz^n \in \mathcal S$ for which the infimum
$
r_N:=\inf\{|p_N(z)| : z = e^{it},\, 0\le t\le2\pi\}
$
is attained.
\end{problem*}

Let us call $r_N$ the \emph{Koebe radius.} Obviously, we have $r_N\ge 1/4.$
In this artice a new interesting family of typically real polynomials is introduced. 
We conjecture that they are univalent (this is proven for degree $\le 6$) and that they attain the value of the Koebe radius. 

The following statement  a central in the paper:

\begin{conj} The value of the Koebe radius for the polynomials of degree $N$ is $ \frac14\sec^2\frac{\pi}{N+2}.$
\end{conj}


\section{Suffridge polynomials}

A natural approach to  Dimitrov's problem would be to look at truncations of the Koebe function. 
However, there is a significant difference between extremal analytic functions and polynomials. 
Since the derivative of a function univalent in $\mathbb D$ has roots  outside $\mathbb  D$, Vieta's theorem implies the estimate on the leading coefficient
\begin{equation}\label{aN}
|a_N|\le\frac1N.
\end{equation}
Since the coefficients of the Koebe function increase, the truncation is not a univalent in $\mathbb D$ polynomial.

Unfortunately, the variety of known polynomials univalent in $\mathbb D$ is quite limited. 
The estimate \eqref{aN} suggests considering the polynomials $  A_N(z)=\sum_{k=1}^N\frac 1k{z^k}$. 
These are partial sums of the  function $-\log(1-z)$ which is univalent in $\mathbb D$. 
They were proven to be univalent in $\mathbb D$ by G.~W. Alexander in the milestone paper \cite{A}. 
For these polynomials, $|A_N(-1)|\ge \frac12$ and $\frac12$ is sharp. 

Other popular extremal polynomials satisfying \eqref{aN} are the Fej\'er polynomials 
$$ 
F_N(z)=\sum_{k=1}^N\biggl(1-\frac {k-1}N\biggr) {z^k}.
$$ 
These  again indicate that the constant $\frac12$ might be sharp in general. 
Certainly, we need more polynomials to test. However, to construct new extremal univalent polynomials is a quite challenging task. 

Returning to the Koebe function, we should recall that it is extremal for the Bieberbach conjecture and has increasing coefficients, while the coefficients in the above examples are decreasing. 
A  powerful idea of Ted Suffridge \cite{S} was to multiply the Fej\'er coefficients by the sine factor  $\sin\frac{\pi k}{N+1}$, making the new coefficients increase up to some level.  He introduced a remarkable family of extremal polynomials
$$
S_{N,j}(z) = \sum^N_{k=1}
\biggl(1 -\frac{k-1}N\biggr) \frac{\sin(\pi kj /(N + 1))}{\sin(\pi j /(N + 1))}z^k, \qquad j=1,\ldots ,N,
$$
which turn \eqref{aN} into equality. He proved that they are univalent in $\mathbb D$. Below instead of $S_{N,1}(z)$ we will simply write $S_{N}(z)$. 

Also, Suffridge showed that whenever $p_N(z)$ is a polynomial in $\mathcal S$ with real coefficients and $|a_N|=1/N$, the remaining coefficients of $p_N(z)$ are also dominated by the coefficients of  $S_{N}(z)$. 

Moreover,   
$$
|S_{N}(-1)|=\frac14\frac{N+1}N\sec^2\frac\pi{2(N+1)}\to \frac14,
$$ 
hence these polynomials indicate that $1/4$ is asymptotically sharp for the polynomial version of the Koebe Quarter Theorem (cf. \cite{CR}). Thus, Suffridge polynomials may be considered as a counterpart of the Koebe function. 

Note that the value $|S_{N}(-1)|$ is the smallest distance from the image of the unit circle to the origin for polynomials $S_{N}(z)$, but only for even degree. 
For  polynomials of odd degree the infimum $\inf\{|S_{N}(z)|:|z|=1\}$ is not achieved at  $z=-1$, but at  a different point $\xi$ such that $S^\prime_{N}(\xi)=0$ \cite{DDS} (see Fig. 1).

\centerline{
\includegraphics[scale=0.25]{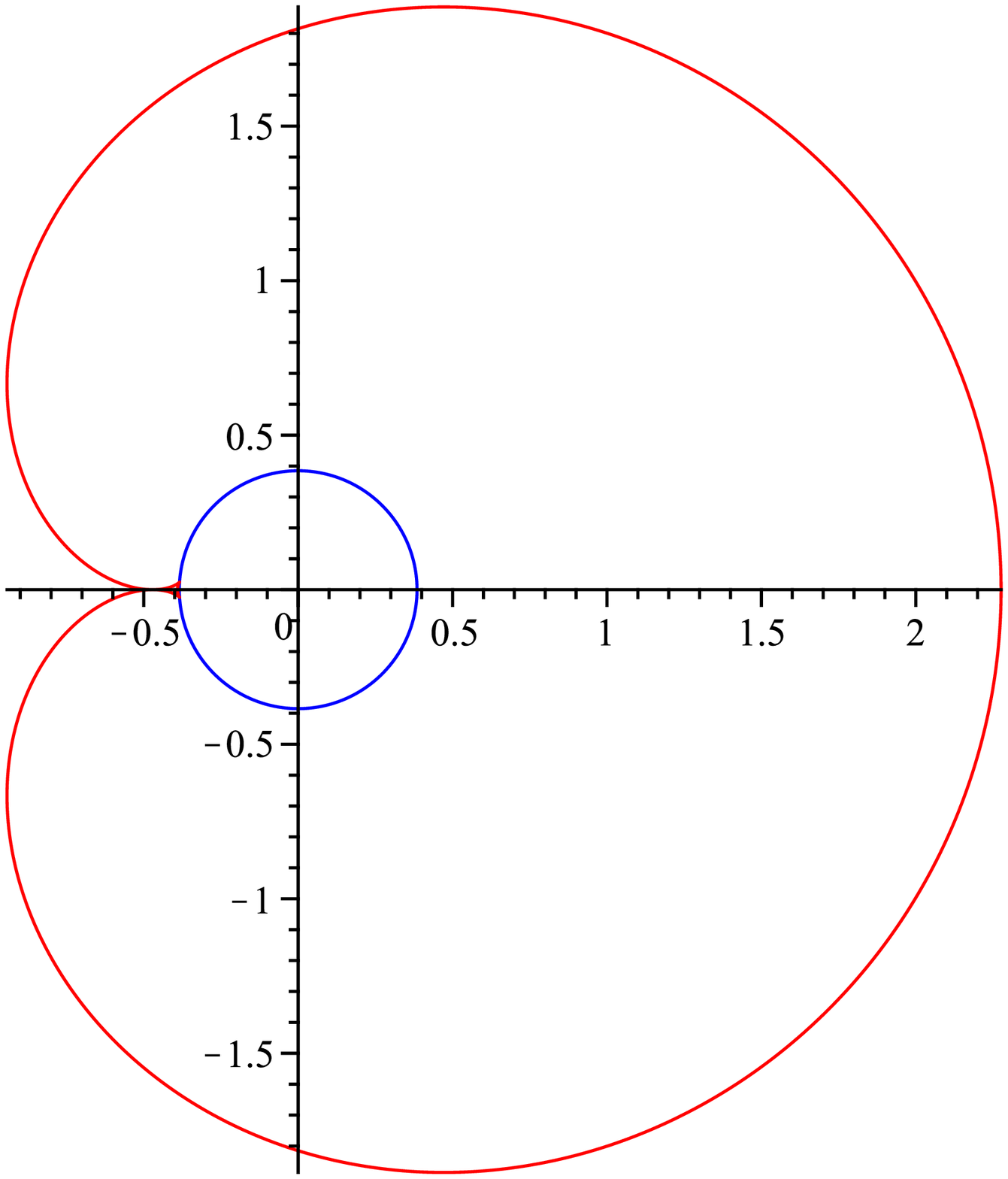}
\includegraphics[scale=0.25]{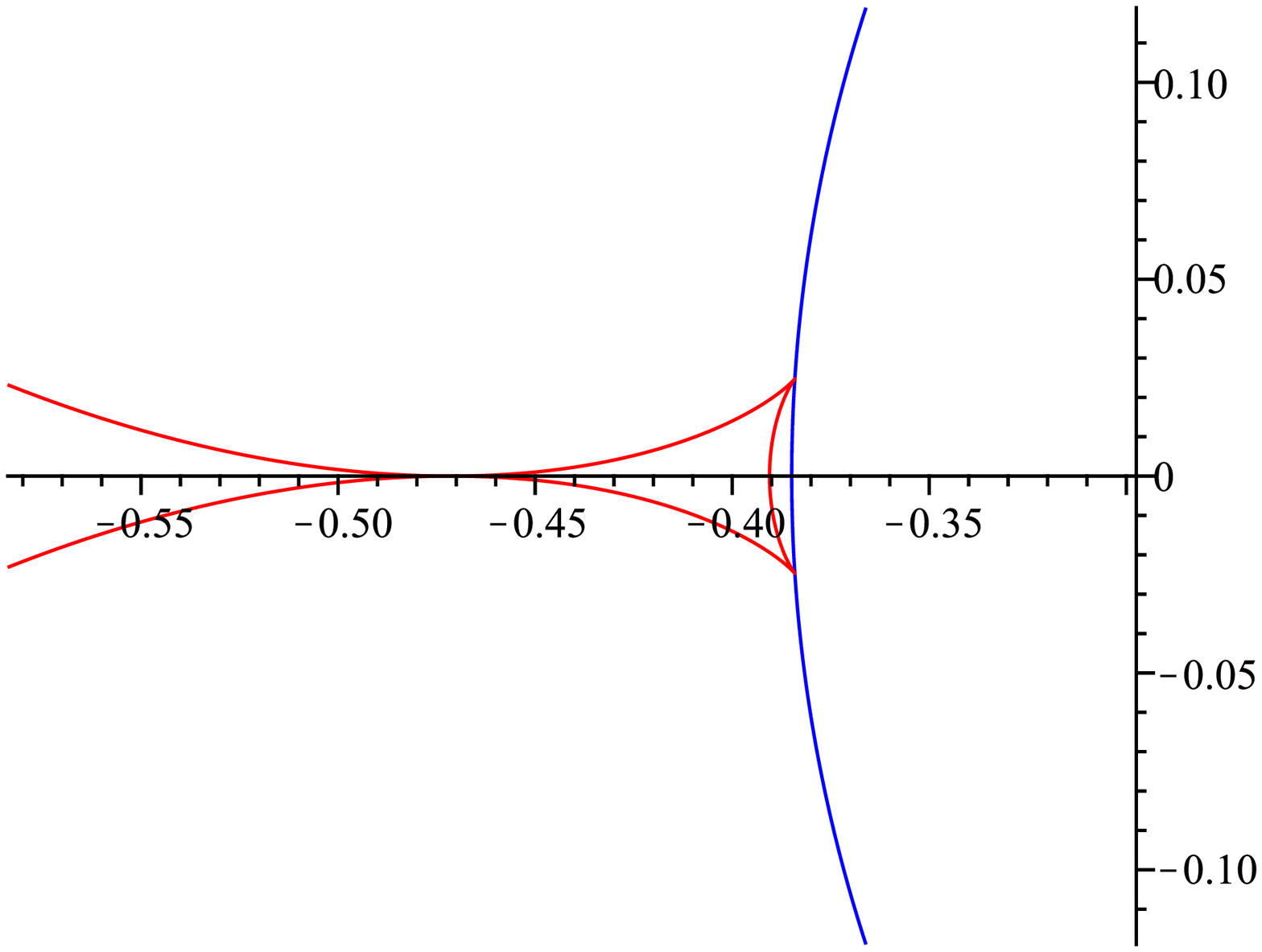}}
\centerline{Fig 1: The image and fragment for $S_{3}(\mathbb D).$}
\medskip

Also, note that the roots of the derivative of a Suffridge polynomial are on $\partial\mathbb D$, and the leading coefficient is $1/N$---the extremal case of the univalence property. 
What could be a better candidate to be a solution to Dimitrov's problem? Actually, Dimitrov \cite[p.~16]{D} asked a specific question about the Suffridge polynomial $S_N(z)$: {\it  Is it extremal for every fixed $N$}? Note that it is indeed extremal for $N=1,2$.

However, our numerous attempts to confirm Dimitrov's conjecture have failed. 
We have got a growing feeling that maybe some other polynomials  could beat out the Suffridge ones. 
But how to get them? 


\section{{ New polynomials}}

We have analyzed the way the Suffridge polynomials emerged. 
They came out as multiplier operators with some sine multipliers applied to the Fej\'er polynomials.
In  turn, the Fej\'er polynomials arose as a solution of the following extremal problem. 

Let $f_N(t)=1+a_1\cos t+\cdots +a_N\cos Nt\ge 0$. Then  $f_N(t)\le N+1$ and the Fej\'er polynomials are extremal here.
Their coefficients can be computed  from the Fej\'er--Riesz representation 
$$
\Phi_N(t)=\frac1{N+1}\left|\sum_{k=0}^N e^{ikt} \right|^2=
1+2\sum_{k=1}^N\left( 1-\frac k{N+1}\right)\cos kt.
$$
Furthermore, for the trigonometric polynomial $ 
F_N(t)=1+2\sum_{k=1}^Na_k\cos{kt}
$ the following Fej\'er inequality is valid:  $  |a_1|\le2\cos\frac\pi{N+2}$, and here the extremal polynomials are the Eg\'ervary--Sz\'asz polynomials \cite{ES}
$$
E_N(t)=\frac{N+2}2+\sum_{k=1}^N\left(\left(N+1-k\right)\cos\frac{k\pi}{N+2} +\frac{\sin\frac{\pi(k+1)}{N+2}}{\sin\frac\pi{N+2}}\right) \cos kt.
$$ 
They can be written in the following form \cite{B}:
$$
E_N(t)=\frac2{N+2}\left|\sum_{k=0}^N\sin\frac{\pi(k+1)}{N+2}e^{ikt}\right|^2=\sum_{k=0}^N b_k\cos kt,
$$
where $b_0=1$ and for $k=1,\ldots ,N$,
\begin{equation}\label{bk}
b_k=\frac{\left(N-k+3\right)\sin\frac{\left(k+1\right)\pi}{N+2}-
\left(N-k+1\right)\sin\frac{\left(k-1\right)\pi}{N+2}}{(N+2)\sin\frac\pi{N+2}}.
\end{equation}

Now, let us apply the same approach  to the Egerv\'ary--Sz\'asz polynomials, i.e. multiply the coefficients \eqref{bk} by $  \sin\frac{k\pi}{N+1}$ and introduce the new polynomials
\begin{equation}\label{P}
P_N(z)=\sum_{k=1}^N b_k \sin\frac{\pi k}{N+1}z^k.
\end{equation}
Below are  some examples: 
\begin{align*}
P_1(z)&=z,\quad
P_2(z)=z+\frac12z^2,\quad
\\
P_3(z)&=z+\frac2{\sqrt5}z^2+\frac12\left(1- \frac1{\sqrt5}\right)z^3,\quad P_4(z)=z+\frac76z^2+\frac23z^3+\frac16z^4,
\\
P_5(z)&=
\begin{aligned}[t]
z&+{\frac {8-40\, \left( \cos \left( \pi/7 \right)  \right) ^{2}+32\,
 \left( \cos \left( \pi/7 \right)  \right) ^{3}-24\,\cos \left( \pi/7
 \right) }{40\, \left( \cos \left( \pi/7 \right)  \right) ^{3}-30\,
\cos \left( \pi/7 \right) -32\, \left( \cos \left( \pi/7 \right) 
 \right) ^{2}+7}}{z}^{2} \\
&+{\frac {24\, \left( \cos \left( \pi/7
 \right)  \right) ^{3}-28\, \left( \cos \left( \pi/7 \right)  \right) 
^{2}-18\,\cos \left( \pi/7 \right) +4}{40\, \left( \cos \left( \pi/7
 \right)  \right) ^{3}-30\,\cos \left( \pi/7 \right) -32\, \left( \cos
 \left( \pi/7 \right)  \right) ^{2}+7}}{z}^{3}\\
&+{\frac {16\, \left( 
\cos \left( \pi/7 \right)  \right) ^{3}-16\, \left( \cos \left( \pi/7
 \right)  \right) ^{2}-12\,\cos \left( \pi/7 \right) +4}{40\, \left( 
\cos \left( \pi/7 \right)  \right) ^{3}-30\,\cos \left( \pi/7 \right) 
-32\, \left( \cos \left( \pi/7 \right)  \right) ^{2}+7}}{z}^{4}\\
&+{\frac {8\, \left( \cos \left( \pi/7 \right)  \right) ^{3}-4\, \left( 
\cos \left( \pi/7 \right)  \right) ^{2}-6\,\cos \left( \pi/7 \right) +
1}{40\, \left( \cos \left( \pi/7 \right)  \right) ^{3}-30\,\cos
 \left( \pi/7 \right) -32\, \left( \cos \left( \pi/7 \right)  \right) 
^{2}+7}}{z}^{5},
\end{aligned}
\\
P_6(z)&=
z+{\frac {9+8\,\sqrt {2}}{4\,\sqrt {2}+8}}{z}^{2}+{\frac {6\,\sqrt {2
}+10}{4\,\sqrt {2}+8}}{z}^{3}+{\frac {4\,\sqrt {2}+6}{4\,\sqrt {2}+8}}
{z}^{4}\\
&\quad+{\frac {2\,\sqrt {2}+2}{4\,\sqrt {2}+8}}{z}^{5}+ \frac1{4\,
\sqrt {2}+8}{z}^{6}
 \end{align*}

\begin{theorem}   
The following representation is valid for $t\in(0,\pi),\; t\not=\frac{2\pi}{N+2}$: 
\begin{align}\label{PN}
P_N(e^{it})&=
  \frac{1}{2\bigl( \cos t- \cos\frac{2\pi}{N+2}\bigr)}\\
 &\quad+
\frac{1-\cos\frac{2\pi}{N+2}}{(N+2)(1-\cos t)}\,
\frac{\sin t\sin\frac{N+2}{2}t}{\bigl(\cos t-\cos\frac{2\pi}{N+2}\bigr)^2} e^{\frac{N+2}{2}it}.\notag
\end{align}
\end{theorem}

\begin{proof}
We begin with
\begin{multline*}
P_N(z)=
 \frac1{ \left(N+2\right)\sin\frac{2\pi}{N+2}}\sum_{k=1}^N\bigg[\left(N-k+3\right)\sin\frac{\left(k+1\right)\pi}{N+2}\\
-\left(N-k+1\right)\sin\frac{\left(k-1\right)\pi}{N+2}\bigg]\sin\frac{k\pi}{N+2}z^k.
\end{multline*}
Having in mind that
$$
\bigg[2\sin(\pi)-0\cdot \sin\frac{N\pi}{N+2}\bigg]\sin\frac{\left(N+1\right)\pi}{N+2}z^{N+1} \equiv 0
$$
we can  sum  to $N+1$. A further modification produces
$$
P_N(z)=
 \frac1{ \left(N+2\right)\sin\frac{2\pi}{N+2}}\sum_{k=1}^{N+1}\bigg[\left(N-k+2\right)\sin\frac{2k\pi}{N+2}+
2\cdot\frac{\cos\frac{\pi}{N+2}}{\sin\frac{\pi}{N+2}}\sin^2\frac{k\pi}{N+2}\bigg]z^k.
$$
An important observation is that
$$
\frac{N+1}{N+2}\cdot S_{N+1,2}(z)=\frac1{ \left(N+2\right)\sin\frac{2\pi}{N+2}}\sum_{k=1}^{N+1}\left(N-k+2\right)\sin\frac{2k\pi}{N+2}\cdot{z^k},
$$
where $S_{N+1,2}(z)$ is the second Suffridge polynomial of order $N+1$. By using formula (5) in \cite[p. 496]{S} for $n=N+1$ and $j=2$ we get
\begin{align*}
\frac{N+2}{N+1}\cdot S_{N+1,2}\left(e^{it}\right)=\frac{1}{2\bigl(\cos t - \cos\frac{2\pi}{N+2}\bigr)}+
\frac{1}{N+2}\cdot \frac{\sin t \cdot \sin\frac{N+2}{2}t}{\bigl(\cos t-\cos\frac{2\pi}{N+2}\bigr)^2}\cdot e^{\frac{N+2}{2}it}.
\end{align*}
Meanwhile,
$$
\sum_{k=1}^{N+1}\sin^2\frac{k\pi}{N+2}e^{ikt}=\sin^2\frac{\pi}{N+2}\cdot\frac{\sin\frac{N+2}{2}t}{\cos{t}-\cos\frac{2\pi}{N+2}}\cdot
\frac{\sin t}{1-\cos t}\cdot e^{i\frac{N+2}{2}t}.
$$
By combining both formulas, we get the formula in the theorem.
\end{proof}

Note that the right hand side of \eqref{PN} has removable singularities, thus it is  in fact a trigonometric polynomial.

\begin{theorem}\label{thm2}   
The following representation is valid for $t\in(0,\pi),\; t\not=\frac{2\pi}{N+2}$:
\begin{multline*}
4|P_N(e^{it})|^2\\
  \begin{aligned}
&=\left(\frac{\cos\frac{N+2}{2}t}{\cos t- \cos\frac{2\pi}{N+2}}+\frac2{N+2}\frac{1-\cos\frac{2\pi}{N+2}}{1-\cos t}
\frac{\sin t}{(\cos t- \cos\frac{2\pi}{N+2})^2}\sin\frac{N+2}{2}t\right)^2\\
&\quad+\left(\frac{\sin\frac{N+2}{2}t}{\cos t- \cos\frac{2\pi}{N+2}} \right)^2.
\end{aligned}
\end{multline*}
\end{theorem}

The Theorem 2 can be directly verified by tedious standard computations.\\

Further, in order to better understand the behaviour of $P_N$, we will pull back its norm to $\mathbb{R}^+_0$ via the Weierstrass map.
The pull back, 
\[ R_N (x)=|P_N(e^{it})|^2 |_{t=2 \arctan x}\] 
allows us to study a single period of the function.

\begin{theorem}\label{thm3}  
If $(R_N(x))^\prime<0$ for all $x\in(0,\infty)$ then the polynomial $P_N(z)$ is univalent in $\mathbb D$ and 
\begin{equation}\label{PN-1}
|P_N(-1)|=\frac14\sec^2\frac{\pi}{N+2}.
\end{equation}
\end{theorem}  

\begin{proof} 
Indeed, taking the imaginary part we get
\begin{equation}\label{ImPN}
\Im(P_N(e^{it}))=\frac{1-\cos\frac{2\pi}{N+2}}{(N+2)(1-\cos t)}\,
\frac{\sin t \left(\sin\frac{N+2}{2}t\right)^2}{\bigl(\cos t-\cos\frac{2\pi}{N+2}\bigr)^2}.
\end{equation}
Since $\Im(P_N(e^{it}))\ge0$ on $[0,\pi]$, the monotonicity of $R_N(x)$ in $x$ implies the monotonicity of $|P_N(e^{it})|$ in $t$ which implies that $P_N(z)$  takes no value more than once on $\partial \mathbb D$,
thus the polynomial $P_{N}(z)$ is univalent in $\mathbb D$ (cf. \cite[6.4.5, p.~201]{T}). Furthermore,
$$
P_N(-1)=\frac{1}{2\bigl( \cos \pi- \cos\frac{2\pi}{N+2}\bigr)}=-\frac14\sec^2\frac{\pi}{N+2}.
$$
\end{proof}

\section{Univalence for small N}

Our first observation is that $R_N(x)$ is of the form  
\begin{equation}\label{eq:form}
 \frac{T_N(x)}{(1+x^2)^{N-1}}
 \end{equation}
where $T_N$ is an even polynomial of degree $2(N-1)$.
Indeed, one shows by induction that $\cos (n \arctan x)$ (or $\sin (n \arctan x)$) is a rational function of the form
\[ \frac{c_n(x)}{(1+x^2)^{\frac n 2}} \]
where $c_n$ is a polynomial.
An ugly but elementary computation implies then that $R_N$ is of the form shown in equation~\eqref{eq:form}. \\

The benefit is that the monotonicity can now checked by a deterministic algorithm: by using a Sturm sequence one can count the real roots of the numerator of the derivative of $R_N(x)$:
\[ (1+x^2)^{N-2} \Delta_N(x):=(1+x^2)^{N-2} \left( T_N'(x)(1+x^2) - 2 (N-1) x T_N(x)\right). \] 
This allows us to determine the univalence of the function $P$:


\subsection{The case $N=1$}
In this case $T_1(x)=4$, thus the Koebe radius is 1. 


\subsection{The case $N=2$}
In this case $T_2(x)=9+x^2$, and the Koebe radius $r_2$ is $|P_2(-1)|=1/2.$


\subsection{The case $N=3$}
In this case the polynomial $P_3(z)$ is univalent again.
\begin{align*}
T_3(x)&=
-\frac {2}{5} \left(-27-9\sqrt{5} - (18+10 \sqrt 5 ) x^2 - (35 - 15 \sqrt 5 ) x^4\right) 
\end{align*}

As $\Delta_3$ is quadratic, is easy to check that $R_3$ is decreasing. on $\IR^+$.
This implies the estimate $r_3\le|P_3(-1)|=\frac{3-\sqrt 5}2= 0.382\ldots$ for the Koebe radius. Note that for the Suffridge polynomial we have  $|S_{3}(-1)|=0.3905\ldots$ and the minimal distance from the image of the unit circle to the origin is $0.3849\ldots$ \cite{DDS}. These estimates imply a negative answer to Dimitrov's question for cubic polynomials.


\subsection{The case $N=4$}
In this case the polynomial $P_4(z)$ is univalent (see \cite{J}). We have
\begin{align*}
T_4(x)=\frac{4}{9} \left(x^2+9\right) \left(x^4-2 x^2+9\right)
\end{align*}
(This can also be seen from $\Delta$ which is biquadratic)
The discriminant is $-37.13\ldots,$ therefore the smallest value for $R_4(x)$ is at $-1$,
 which implies  $r_4\le|P_4(-1)|=1/3$.
 

\subsection{The case $N=5$}

\begin{align*}
T_5(x)&=
   -49 x^8 \left(121 \sin \left(\frac{\pi }{14}\right)-42 \left(3+5 \sin \left(\frac{3 \pi
   }{14}\right)\right)+55 \cos \left(\frac{\pi }{7}\right)\right)
   \\
&\quad + 4 x^6 \left(8924-9107 \sin
   \left(\frac{\pi }{14}\right)+15094 \sin \left(\frac{3 \pi }{14}\right)-4507 \cos \left(\frac{\pi
   }{7}\right)\right)\\
&\quad +2 x^4 \left(84326-20935 \sin \left(\frac{\pi }{14}\right)+116342 \sin
   \left(\frac{3 \pi }{14}\right)+33443 \cos \left(\frac{\pi }{7}\right)\right)\\
&\quad -4 x^2 \left(37328+61431
   \sin \left(\frac{\pi }{14}\right)+31802 \sin \left(\frac{3 \pi }{14}\right)+133139 \cos
   \left(\frac{\pi }{7}\right)\right) \\
&\quad +21 \left(22702+3859 \sin \left(\frac{\pi }{14}\right)+30218 \sin
   \left(\frac{3 \pi }{14}\right)+31141 \cos \left(\frac{\pi }{7}\right)\right) \\
   &\quad  \times \frac{1}{784
    \left(\sin \left(\frac{3 \pi }{28}\right)+\cos \left(\frac{3 \pi
   }{28}\right)\right)^{14}}\\
\end{align*}
and 
\begin{align*}
\frac {\Delta}{16x} &=
   x^6 \left(-1375+1589 \sin \left(\frac{\pi }{14}\right)-2402 \sin \left(\frac{3 \pi
   }{14}\right)+906 \cos \left(\frac{\pi }{7}\right)\right) \\
   &\quad -x^4 \left(28777+3193 \sin \left(\frac{\pi
   }{14}\right)+35530 \sin \left(\frac{3 \pi }{14}\right)+23482 \cos \left(\frac{\pi
   }{7}\right)\right)\\
   &\quad +x^2 \left(98155+81679 \sin \left(\frac{\pi }{14}\right)+105874 \sin \left(\frac{3
   \pi }{14}\right)+216430 \cos \left(\frac{\pi }{7}\right)\right)\\
   &\quad-5 \left(51407+14247 \sin
   \left(\frac{\pi }{14}\right)+66638 \sin \left(\frac{3 \pi }{14}\right)+78710 \cos \left(\frac{\pi
   }{7}\right)\right)
   \end{align*}

 Again $P_5(z)$ is univalent and this gives us an estimate on the Koebe radius $r_5\le|P_5(-1)|=0.3080\ldots.$ 


\subsection{The case $N=6$} 
In this case
\begin{align*}
T_6(x) &=\left(6-4 \sqrt{2}\right) x^{10}+\left(246-172 \sqrt{2}\right) x^8+4 \left(70
   \sqrt{2}-99\right) x^6 \\
   &\quad -4 \left(30 \sqrt{2}-61\right) x^4-10 \left(9+2 \sqrt{2}\right)
   x^2+36 \sqrt{2}+54
\end{align*}
and
\begin{align*}
\frac {\Delta}{16 x} &= \left(19 \sqrt{2}-27\right) x^8+\left(222-156 \sqrt{2}\right) x^6 \\
   &\quad+ \left(150
   \sqrt{2}-240\right) x^4+ \left(106-20 \sqrt{2}\right) x^2+ \left(-45-25 \sqrt{2}\right)
\end{align*}

This is the last situation where we can find the roots exactly and this implies the estimate for the Koebe radius $r_6\le|P_6(-1)|=0.2929\ldots.$ We conjecture that the estimates obtained are in fact true values.

\subsection{Larger $N$}

As mentioned above, by using the Weierstrass transform, univalence follows from the study of the roots of $\Delta_N$. We used Mathematica that check exactly that $\Delta(x)$ has no real roots outside $0$ for $N$ up to $51$ and the growth of $R_N$ can be checked exactly for any $N$. 

\section{Koebe radius for typically real polynomials}
Let us recall that a polynomial $p(z)$ with real coefficients  is \emph{typically real} in $\mathbb D$ if $\Im(p(z))\Im(z)\ge 0$ for $z\in \mathbb D$. Since $\Im(P_N(e^{it}))\ge 0$ on $[0,\pi]$, the polynomials \eqref{P} are typically real and  the formula \eqref{PN-1} implies the following estimate on the Koebe radius for typically real polynomials:
\begin{equation}\label{KN}
r_N\le \frac14\sec^2\frac{\pi}{N+2}.
\end{equation}

This estimate may be complemented by an estimate from below. {
Indeed, in 1916 Ludwig Bieberbach} proved the estimate 
\begin{equation}\label{a2}|a_2|\le 2\end{equation} 
for the the second Taylor coefficient of a function from ${\mathcal S}.$ 
This estimate implies  the Koebe conjecture by the following beautiful argument. Let $f\in \mathcal S$, $f(z)=z+\alpha_2z^2+\cdots$ and $\gamma\not\in f(\mathbb D)$. Then
\begin{align*}
\frac{f(z)}{1-\frac{f(z)}\gamma} &=f(z)\biggl(1+\frac{f(z)}\gamma+\cdots \biggr)=
                                   (z+\alpha_2z^2+\cdots )\biggl(1+\frac z\gamma +\cdots \biggr)\\
  &=z+\biggl(a_2+\frac1\gamma\biggr)z^2+\cdots.
  \end{align*}
By the Bieberbach estimate we have
$|\alpha_2+\frac1\gamma|\le 2$, hence  
$|\gamma|\ge\frac1{2+|\alpha_2|}$,
and again by~\eqref{a2} we get
$ 
|\gamma|\ge\frac14
$
which is a statement of Koebe Theorem.

The above argument implies that
$  r_N\ge\frac1{2+\sup|a_2|}$. W.~Rogosinski and G.~Szeg\"o \cite{RS}  got an estimate for the second coefficient of a typically real polynomial, $|a_2|\le2\cos2\psi_N$, where $\psi_N=\pi/(N+3)$ if $N$ is odd, and $\psi_N$ is the smallest positive root of the equation
$
(N+4)\sin\,(N+2)\psi_N+(N+2)\sin\,(N+4)\psi_N=0
$ 
if $N$ is even. Since  univalent polynomials with real coefficients are typically real, we get an estimate on the Koebe radius for univalent  polynomials:
\begin{equation}\label{KN1}
r_N\ge\frac14\sec^2\psi_N. 
\end{equation}


\section*{Acknowledgements} {The authors are grateful to Konstantin Dyakonov and Paul Hagelstein for numerous fruitful discussions, useful comments and observations. They also would like to thank Adhemar Bultheel and Plamen Iliev for their time and interesting exchanges and to thank Jerzy Trzecjak for the help in preparation of manuscript.}

\end{document}